\documentclass{mathincs}
\usepackage[T1]{fontenc}
\usepackage[latin9]{inputenc}
\usepackage{amstext}
\usepackage{amsthm}
\usepackage{amssymb}
\usepackage{graphicx}
\usepackage{float}
\usepackage{xcolor}
\usepackage{url}
\usepackage{mma}

\makeatletter
\numberwithin{equation}{section}
\numberwithin{figure}{section}
  \theoremstyle{plain}
  \newtheorem{thm}{\protect\theoremname}[section]
  \theoremstyle{definition}
  \newtheorem{defn}[thm]{\protect\definitionname}
  \theoremstyle{remark}
  \newtheorem{rem}[thm]{\protect\remarkname}
  \theoremstyle{plain}
  \newtheorem{prop}[thm]{\protect\propositionname}
  \theoremstyle{definition}
  \newtheorem{example}[thm]{\protect\examplename}
  \theoremstyle{plain}
  \newtheorem{lemma}[thm]{Lemma}
  \theoremstyle{plain}
  \newtheorem{conj}[thm]{Conjecture}

\DeclareMathOperator{\diag}{diag}
\DeclareMathOperator{\Tr}{tr}
\newcommand{\SPD}{\operatorname{SPD}}

\definecolor{tan}{rgb}{0.82, 0.71, 0.55}

\makeatother

\usepackage[american]{babel}
\providecommand{\definitionname}{Definition}
\providecommand{\examplename}{Example}
\providecommand{\propositionname}{Proposition}
\providecommand{\remarkname}{Remark}
\providecommand{\theoremname}{Theorem}

\begin{document}

\title{Calculation and Properties of Zonal Polynomials%
  \footnote{This is a post-peer-review, pre-copyedit version of an article published in
    Mathematics in Computer Science. The final authenticated version is available online
    at: http://dx.doi.org/10.1007/s11786-020-00458-0}}

\author{Lin Jiu}
\address{%
Department of Mathematics and Statistics,Dalhousie University, 
6316 Coburg Road, Halifax, Nova Scotia, Canada B3H 4R2 \\
Lin.Jiu@dal.ca
}
\author{Christoph Koutschan}
\address{%
Johann Radon Institute for Computational and Applied Mathematics (RICAM),
Altenberger Stra\ss e 69, A-4040 Linz, Austria \\
christoph.koutschan@ricam.oeaw.ac.at
}

\thanks{%
LJ was supported by the Austrian Science Fund (FWF): P29467-N32.\\
CK was supported by the Austrian Science Fund (FWF): P29467-N32 and F5011-N15.
}



\begin{abstract}
  We investigate the zonal polynomials, a family of symmetric polynomials that
  appear in many mathematical contexts, such as multivariate statistics,
  differential geometry, representation theory, and combinatorics.  We present
  two computer algebra packages, in SageMath and in Mathematica, for their
  computation. With the help of these software packages, we carry out an experimental
  mathematics study of some properties of zonal polynomials. Moreover, we
  derive and prove closed forms for several infinite families of zonal
  polynomial coefficients.
\end{abstract}

\maketitle

\vspace{-8pt}

\section{Introduction}
At the beginning of our study, we recall the generalized hypergeometric
function ${}_pF_q$, defined as the infinite series
\begin{equation}\label{eq:pFq}
  {}_pF_q\left(\genfrac{}{}{0pt}{}{a_{1},\ldots,a_{p}}{b_{1},\ldots,b_{q}}\,\bigg|\,z\right) :=
  \sum_{n=0}^{\infty}\frac{(a_{1})_{n}\cdots(a_{p})_{n}}{(b_{1})_{n}\cdots(b_{q})_{n}}\cdot\frac{z^{n}}{n!},
\end{equation}
where for positive integer $m$,
$(a)_{m}:=a(a+1)\cdots(a+m-1)$ is the
Pochhammer symbol.  What is less well-known is a remarkable generalization of
this hypergeometric function of a matrix argument, as follows.
\begin{defn}
Given an $m\times m$ symmetric, positive-definite matrix $Y$, the
hypergeometric function ${}_pF_q$ of matrix argument $Y$ is defined as 
\begin{equation}\label{eq:MatrixpFq}
  {}_{p}F_{q}\left(\genfrac{}{}{0pt}{}{a_{1},\ldots,a_{p}}{b_{1},\ldots,b_{q}}\,\bigg|\,Y\right) :=
  \sum_{n=0}^{\infty}\sum_{\lambda\in\mathcal{P}_{n}}\frac{(a_{1})_{\lambda}\cdots(a_{p})_{\lambda}}
      {(b_{1})_{\lambda}\cdots(b_{q})_{\lambda}}\cdot\frac{\mathcal{C}_{\lambda}(Y)}{n!},
\end{equation}
where 
\begin{itemize}
\item $\mathcal{P}_{n}$ is the set of all integer partitions of $n$, in which,
every partition $\lambda\in\mathcal{P}_{n}$ is defined to be a tuple
$\lambda=(\lambda_{1},\ldots,\lambda_{k})$ such that 
$\lambda_{1}\geq\lambda_{2}\geq\cdots\geq\lambda_{k}\geq 1$ and 
$ \lambda_{1}+\cdots+\lambda_{k}=n$;
\item $(a)_{\lambda}$ is the generalized Pochhammer symbol, defined as
\[
(a)_{\lambda}=(a)_{(\lambda_{1},\dots,\lambda_{k})}:=\prod_{i=1}^{k}\left(a-\frac{i-1}{2}\right)_{\!\!\lambda_{i}}\!\!;
\]
\item and finally~$\mathcal{C}_{\lambda}(Y)$ denotes the zonal polynomial
of~$Y$, indexed by a partition~$\lambda$, which is a symmetric homogeneous polynomial
of degree~$n$ (see Section~\ref{sec:DEFS}) in the eigenvalues
$y_{1},\ldots,y_{m}$ of~$Y$, satisfying 
\begin{equation}\label{eq:TrZonal}
  \sum_{\lambda\in\mathcal{P}_{n}}\mathcal{C}_{\lambda}(Y)=(\Tr Y)^{n}=(y_{1}+\cdots+y_{m})^{n}.
\end{equation}
\end{itemize}
\end{defn}

By noting that the zonal polynomial is zero whenever~$k$, the number of parts
of~$\lambda$, exceeds the dimension of~$Y$ (see Remark~\ref{rem:polyzero}),
one recognizes that \eqref{eq:MatrixpFq} indeed specializes to \eqref{eq:pFq}
when $Y$ is a $(1\times1)$-matrix, since for each~$n$, only the partition
$\lambda=(n)$ contributes.

The hypergeometric function of a matrix argument is used in multivariate
statistics, in connection with the Wishart distribution~\cite{Wishart28}; see
Section~\ref{sec:Wishart} and \cite{GrossRichards87} for an introduction.  For
example, the extreme eigenvalues of random matrices can be expressed in terms
of this hypergeometric function~\cite{ButlerPaige11,Johnstone01}.  However,
the numerical evaluation of ${}_pF_q$ functions of a matrix argument is a
notorious problem in multivariate distribution theory~\cite{KoevEdelman06}.

Recent progress on the numerical evaluation is based on the holonomic
gradient method~\cite{NakayamaEtAl11}. In the case $p=q=1$, it is known that
the hypergeometric function $_1F_1(a;c;Y)$ satisfies a holonomic system of
partial differential equations~\cite{Muirhead70} in the variables
$y_1,\dots,y_m$, whose holonomic rank is~$2^m$.  Hashiguchi
et al.~\cite{Hashiguchi} use it to study the cumulative distribution of the
largest eigenvalue of a Wishart matrix. The problem of specializing this
high-dimensional holonomic system to singular regions has been addressed
in~\cite{Noro16}.  As an application, the Wishart distribution arises in the
performance analysis of wireless communication systems under Rayleigh
fading~\cite{SiriteanuEtAl15}. For evaluating the hypergeometric
function of a matrix argument using the holonomic gradient method,
one needs to know the first few zonal polynomials in order to get
accurate initial conditions.

A comprehensive introduction to zonal polynomials~\cite[\S\,35.4]{DLMF} was
given by Takemura~\cite{Takemura}. Interestingly, these polynomials also
appear in completely different mathematical contexts. No direct formula for
their calculation is known, but only partial
results~\cite{KushnerMeisner,Muirhead}. There exist software packages in
Maple~\cite{Stembridge95} and SageMath\footnote{http://doc.sagemath.org/html/en/reference/combinat/sage/combinat/sf/jack.html} to compute with them.

In Section~\ref{sec:DEFS} we give a survey
of different definitions of zonal polynomials. In Section~\ref{sec:calc} we
recall a recursive method by Muirhead~\cite{Muirhead} to calculate zonal
polynomials. Note that there is no general direct formula to obtain the
coefficients of the zonal polynomial.  In the following we will give such
formulas for some special families of zonal polynomial coefficients: In
Section~\ref{sec:zeros} we present conditions under which a
coefficient vanishes, and in Section~\ref{sec:families} we derive closed
forms for the coefficients at the two extremal corners of the coefficient
matrix. Then we give a complete closed form for zonal polynomials in two
variables (Section~\ref{sec:2parts}) and some partial results for three and
four variables (Section~\ref{sec:3+4parts}). Finally, we explain some details
of our software packages, see Sections~\ref{sec:Package} and~\ref{sec:mma},
that were written in the frame of the current work.

\section{\label{sec:DEFS}Definitions of zonal polynomials}

We shall summarize four different definitions of zonal polynomials
involving statistics, differential geometry, representation theory
and combinatorics. Namely, each subsection will present one aspect. 
First of all, we need an important linear space.
\begin{defn}
Let $V_{n}$ be the space of symmetric homogeneous polynomials of
degree $n$ in the variables $y_{1},\ldots,y_{m}$, including the zero polynomial.
Namely, if $f\in V_{n}$, we have 
\begin{itemize}
\item $\deg f=n$ or $f\equiv0$;
\item if $\deg f=n$, then $f$ is symmetric and homogeneous in $y_{1},\ldots,y_{m}$.
\end{itemize}
Moreover, any polynomial $f\in V_{n}$ can also be viewed as a polynomial
in the eigenvalues of an $m\times m$ symmetric, positive-definite matrix~$Y$.
Then, the notations $f(y_1,\dots,y_m)$ and $f(Y)$ are considered equivalent.
Denote the space of $m\times m$ symmetric, positive-definite matrices
by $\SPD(m)$. 
\end{defn}

\subsection{Definition involving the Wishart distribution}\label{sec:Wishart}
The following definitions, claims, and properties in this section
can be found in \cite[pp.~9--22]{Takemura}.
\begin{defn}
Define the elementary symmetric polynomial
\[
  u_{r}(x_{1},\ldots,x_{m}):=\underset{1\leq i_{1}<\cdots<i_{r}\leq m}{\sum}x_{i_{1}}\cdots x_{i_{r}}.
\]
Then, we have a basis for $V_n$: for $\lambda=(\lambda_{1},\ldots,\lambda_{k})\in\mathcal{P}_{n}$, define the polynomials
\[
  \mathcal{U}_{\lambda} :=
  u_{1}^{\lambda_{1}-\lambda_{2}}u_{2}^{\lambda_{2}-\lambda_{3}}\cdots u_{k-1}^{\lambda_{k-1}-\lambda_{k}}u_{k}^{\lambda_{k}}.
\]
\end{defn}
\noindent
Obviously,
$\deg\mathcal{U}_{\lambda} =
  \left(\lambda_{1}-\lambda_{2}\right)+2\left(\lambda_{2}-\lambda_{3}\right)+\cdots+k\lambda_{k} =
  \lambda_{1}+\cdots+\lambda_{k}=n$.
Associate a lexicographical order to $\mathcal{P}_{n}$ as follows:
for $\kappa=\left(\kappa_1,\ldots,\kappa_j\right),\lambda=\left(\lambda_1,\ldots,\lambda_k\right)\in\mathcal{P}_{n}$,
\[
  \kappa>\lambda \quad:\Leftrightarrow\quad
  \kappa_1=\lambda_1\land\cdots\land\kappa_{l-1}=\lambda_{l-1}\land\kappa_{l}>\lambda_{l}
  \quad\text{for some }l.
\]
Then, we can write the basis formed by~$\mathcal{U}_{\lambda}$
as a column vector:
$
\mathcal{U}:=\bigl(\mathcal{U}_{(n)},\allowbreak \mathcal{U}_{(n-1,1)},\ldots,\mathcal{U}_{(1,\ldots,1)}\bigr){}^{T}\!.
$
This subsection presents a definition of $\mathcal{C}_{\lambda}(Y)$
related to the Wishart distribution, defined as follows.
\begin{defn}
Let $X_{\nu\times m}$ be a matrix such that each row is independently drawn
from an $m$-variate normal distribution of mean $0$ and with covariance
matrix~$V$, namely, 
\[
  (x_{i}^{1},\ldots,x_{i}^{m})\sim\mathcal{N}_{m}(0,V)
  \qquad (1\leq i\leq\nu).
\]
Then, we say $S:=X^{T}X$ has the Wishart distribution, denoted by $S=X^{T}X\sim W_{m}(V,\nu)$, where $\nu$ is called the degree of freedom.
\end{defn}
\begin{rem}
Recall the $1$-dimensional case: if $Z_{1},\ldots,Z_{k}\sim\mathcal{N}(0,1)$
are independent Gaussian distributed, then $Q:=Z_{1}+\cdots+Z_{k}\sim\chi_{k}^{2}$.
In other words, the sum of independent Gaussian distributions is chi-square
distributed. Therefore, the Wishart distribution can be viewed as a multi-dimensional
generalization of the chi-square distribution.
\end{rem}
\noindent
Define the linear transform $\tau_{\nu}\colon V_{n}\longrightarrow V_{n}$,
for $Y\in \SPD(m)$, by
\[
  (\tau_{\nu}(\mathcal{U}_{\lambda}))(Y) :=
  \mathbb{E}_{W}[\mathcal{U}_{\lambda}(YW)]\text{ for }W\sim W(I_m,\nu).
\]
As $\mathcal{U}$ forms a basis of $V_{n}$, $\tau_{\nu}(\mathcal{U}):=\left(\tau_{\nu}(\mathcal{U}_{(n)}),\tau_{\nu}(\mathcal{U}_{(n-1,1)}),\ldots,\tau_{\nu}(\mathcal{U}_{(1,\ldots,1)})\right){}^{\!T}$
must be a linear combination of $\mathcal{U}$,
denoted by $\tau_{\nu}(\mathcal{U})=T_{\nu}\mathcal{U}$. 
Properties of the transition matrix $T_{\nu}$ guarantee a diagonalization as $T_{\nu}=\Xi^{-1}\Lambda_{\nu}\Xi$, where 
\begin{itemize}
\item $\Lambda_{\nu}=\diag(2^{n}(\nu/2)_{\lambda})$, for $\lambda\in\mathcal{P}_n$, is the diagonalization of $T_\nu$; 
\item and $\Xi$ is a nonsingular upper triangular matrix, which is uniquely
determined up to a (possibly different) multiplicative constant for
each row.
\end{itemize} 
Now, we can define the zonal polynomials.
\begin{defn}
For $\lambda=(\lambda_1,\ldots,\lambda_k)\in\mathcal{P}_{n}$, the zonal polynomial $\mathcal{Y}_{\lambda}$
is defined by a vector form
\begin{align}
\mathcal{Y}=\left(\mathcal{Y}_{(n)}, \mathcal{Y}_{(n-1,1)}, \ldots, 
\mathcal{Y}_{(1,\ldots,1)}\right)^T
  &= \Xi\,\mathcal{U} \notag \\ 
  &= \Xi\left(\mathcal{U}_{(n)}, \mathcal{U}_{(n-1,1)}, \ldots, \mathcal{U}_{(1,\ldots,1)}
\right)^T\!\!, \label{eq:ZonalY}
\end{align}
and define $\mathcal{C}_{\lambda}(Y)=d_{\lambda}\mathcal{Y}_{\lambda}(Y)$ by the constants $d_\lambda$, given by
\[
  d_{\lambda}=\frac{\underset{i<j}{\prod}(2\lambda_{i}-2\lambda_{j}-i+j)}{\overset{k}{\underset{i=1}{\prod}}(2\lambda_{i}+k-i)!}
  \cdot\frac{2^{n}n!}{(2n)!}.
\]
\end{defn}

\begin{example}
  For $n=4$, $\nu=3$, and $m=2$, we compare an exact computation with a
  Monte-Carlo experiment.  Since $m=2$, we only need to consider partitions
  of~$4$ with at most $2$ parts and get $\mathcal{U}=\left((y_1+y_2)^4,
  y_1y_2(y_1+y_2)^2, y_1^2y_2^2\right)$. 

  The map $\tau_\nu$ is defined to be an expectation. In order to approximate
  $\tau_\nu(\mathcal{U})(Y)$ numerically, we sample a large number of Wishart
  matrices~$W$. In Mathematica, the command
  \begin{mma}
    \In |RandomVariate|[|WishartMatrixDistribution|[3, |IdentityMatrix|[2]]] \\
    \Out \rule{0pt}{8pt}\{\{3.0965, -0.551265\}, \{-0.551265, 1.59861\}\} \\
  \end{mma}
  \noindent
  randomly generates such a $2\times2$ matrix~$W$. For simplicity, let
  $Y=\diag(y_1,y_2)$. The eigenvalues of $YW$ are in general algebraic
  expressions, but after plugging them into the symmetric polynomials
  of~$\mathcal{U}$ and simplifying, one gets polynomials back. For these
  simplifications, it is advisable to employ exact arithmetic instead of
  floating point numbers, and therefore we convert $W$ to have exact rational
  entries at the very beginning. Averaging over $N=10^6$ samples yields the
  following approximation $\sum_{i=1}^N \mathcal{U}(YW_i)/N$ for the vector
  $\tau_\nu(\mathcal{U})(Y)$:
  \[
    \begin{pmatrix}
      945.715\,y_1^4 + 1261.66\,y_1^3y_2 + 1347.86\,y_1^2y_2^2 + 1258.51\,y_1y_2^3 + 947.094\,y_2^4 \\
      210.465\,y_1^3y_2 + 299.699\,y_1^2y_2^2 + 209.768\,y_1y_2^3 \\
      119.769\,y_1^2 y_2^2
    \end{pmatrix}.
  \]

  Now we want to compare this approximate result with the exact one.  For this
  purpose, we proceed ``backwards'', i.e., we start with the zonal polynomials
  $\mathcal{C}_\lambda(Y)$ (their coefficients are given in
  Example~\ref{Tables}).  After dividing them with the constants~$d_\lambda$,
  we can use~\eqref{eq:ZonalY},
  \[
    \Xi \cdot \begin{pmatrix} (y_1+y_2)^4 \\ y_1y_2(y_1+y_2)^2 \\ y_1^2y_2^2 \end{pmatrix} =
    \begin{pmatrix}
      4233600 \cdot \left(y_1^4+y_2^4 + \frac{4}{7}(y_1^3y_2+y_1y_2^3) + \frac{18}{35}y_1^2y_2^2 \right) \\
      211680 \cdot \left(\frac{24}{7}(y_1^3y_2+y_1y_2^3) + \frac{16}{7}y_1^2y_2^2\right) \\
      967680 \cdot y_1^2 y_2^2
    \end{pmatrix}
  \]
  to determine the matrix~$\Xi$. As a result, one obtains
  \[
    \Xi = 120960 \cdot
    \begin{pmatrix}
      35 & -120 & 48 \\
      0 & 6 & -8 \\
      0 & 0 & 8
    \end{pmatrix}
    \quad\text{and}\quad
    T_\nu = \Xi^{-1}\Lambda_{\nu}\Xi =
    \begin{pmatrix}
       945 & -2520 & 720 \\
       0 & 210 & -120 \\
       0 & 0 & 120
    \end{pmatrix},
  \]
  where $\Lambda_{\nu}=\diag(945,210,120)$. Applying this transition matrix
  to the $\mathcal{U}$-basis, one finally gets
  \[
    T_\nu \, \mathcal{U} =
    \begin{pmatrix}
      945\,y_1^4 + 1260\,y_1^3y_2 + 1350\,y_1^2y_2^2 + 1260\,y_1y_2^3 + 945\,y_2^4 \\
      210\,y_1^3y_2 + 300\,y_1^2y_2^2 + 210\,y_1y_2^3 \\
      120\,y_1^2y_2^2
    \end{pmatrix}
  \]
  and sees that the previous Monte-Carlo simulation delivered quite
  accurate results.
\end{example}

\subsection{Definition in differential geometry}

Good references for the material in this subsection are~\cite{Helgason,James1,Moakher}.
We first recall the Laplace-Beltrami operator on Riemannian manifolds.
\begin{defn}
On a Riemannian manifold $(M,g)$, the Laplace-Beltrami
operator on smooth functions $f\in C^{\infty}(M)$ is given by 
\[
\Delta f:=(\mathrm{div}\circ\mathrm{grad})f=\sum_{i,k=1}^{n}\frac{1}{\sqrt{G}}\partial_{k}\left(g^{ik}\sqrt{G}\partial_{i}f\right),
\]
where $n=\dim M$, $(g_{ij})_{n\times n}$ is the metric
matrix, and $G:=\det(g_{ij})$. 
\end{defn}
\begin{rem}
When $M=\mathbb{R}^{n}$ and $(g_{ij})=I_{n}$, the identity
matrix, we have the usual Laplace operator:
\[
\Delta f=\sum_{i=1}^{n}\frac{\partial^{2}f}{\partial x_{i}^{2}}.
\]
Namely, the Laplace-Beltrami operator is the generalization of the Laplace
operator on $\mathbb{R}^{n}$. 
\end{rem}
\begin{prop}
Given $X=HY\!H^T\!\in\SPD(m)$, for some orthogonal matrix $H\in O(n)$,
and $Y=\diag(y_{1},\ldots,y_{n})$ being the diagonalization of $X$, James
\cite[eq.~3.12]{James1} derived the Laplace-Beltrami operator as

\begin{equation}
\Delta=\sum_{i=1}^{m}\left(y_{i}^{2}\frac{\partial^{2}}{\partial y_{i}^{2}}-\frac{m-3}{2}y_{i}\frac{\partial}{\partial y_{i}}+\sum_{j=1,j\neq i}^{m}\frac{y_{i}^{2}}{y_{i}-y_{j}}\cdot\frac{\partial}{\partial y_{i}}\right). \label{eq:LaplaceBeltrami}
\end{equation}
\end{prop}
\begin{rem}
The second term on the right-hand side of \eqref{eq:LaplaceBeltrami} is, up to the constant $(m-3)/2$,
the Euler's operator $\sum_{i=1}^{m}y_{i}\frac{\partial}{\partial y_{i}}$,
which has all symmetric, homogeneous polynomials as its eigenfunctions.
Thus, when considering eigenfunctions of the Laplace-Beltrami operator,
it can be eliminated.
\end{rem}
\begin{defn}
The zonal polynomials $\mathcal{C}_{\lambda}(y_{1},\ldots,y_{m})$
are the eigenfunctions of the operator $\Delta_{Y}$, defined by
\[
  \Delta_{Y}:=\sum_{i=1}^{m}\left(y_{i}^{2}\frac{\partial^{2}}{\partial y_{i}^{2}}+\sum_{j=1,j\neq i}^{m}\frac{y_{i}^{2}}{y_{i}-y_{j}}\cdot\frac{\partial}{\partial y_{i}}\right).
\]
In particular,  for $\lambda=(\lambda_1,\ldots,\lambda_k)\in\mathcal{P}_n$, we have
\[
  \Delta_{Y}\mathcal{C}_{\lambda}(Y)=\bigl(\rho_{\lambda}+(m-1)n\bigr)\cdot\mathcal{C}_{\lambda}(Y),
\]
where
\begin{equation}\label{eq:RHO}
  \rho_{\lambda}:=\sum_{i=1}^{k}\lambda_{i}\left(\lambda_{i}-i\right).
\end{equation}
\end{defn}

\subsection{Definition through representation theory}

Consider the general linear group $G=\mathrm{GL}(m)$ on
$V_{n}$. Define a representation as follows. For $g\in\mathrm{GL}(m)$, $Y\in V_n$, and $\varphi\in\mathrm{GL}(V_n)$, 
\[
\left(g\circ\varphi\right)(Y):=\varphi{\left(g^{-1}Y(g^{-1})^{T}\right)}.
\]
As a representation, the linear space can be decomposed into invariant
subspaces \cite[p.~611]{Representation}
\[
V_{n}=\bigoplus_{\lambda\in\mathcal{P}_{n}}V_{\lambda}.
\]

\begin{defn}
Given $Y\in \SPD(m)$ and $\lambda\in\mathcal{P}_{n}$, define the zonal polynomials by
the projection
\begin{equation}
\mathcal{C}_{\lambda}(Y)=(\Tr Y)^{n}\Big|_{V_{\lambda}}.\label{eq:projection}
\end{equation}
\end{defn}
\begin{rem}
Note that \eqref{eq:projection} confirms \eqref{eq:TrZonal}. 
\end{rem}

\subsection{A short remark on Macdonald, Jack and zonal polynomials}\label{sec:Macdonald}

A limit case (by taking $t=q^\alpha$ and letting $q\rightarrow 1$) of the Macdonald polynomials gives the Jack polynomials $J_{\lambda}^{(\alpha)}$,
which when $\alpha=2$, gives the zonal polynomials $\mathcal{Z}_\lambda(Y)$. $\mathcal{Z}_\lambda(Y)$ differs from $\mathcal{C}_\lambda(Y)$ only by a constant factor. 

\section{Calculation of zonal polynomials}\label{sec:calc}

Although there are several ways to define the zonal polynomial $\mathcal{C}_{\lambda}(Y)$,
in practice, none of these definitions gives an algorithm or formula
to directly compute~$\mathcal{C}_{\lambda}(Y)$. Now, we
follow the steps by Muirhead \cite{Muirhead} to build up packages for the
calculation of~$\mathcal{C}_{\lambda}(Y)$.
\begin{defn}
For $\lambda=(\lambda_1,\ldots,\lambda_k)\in\mathcal{P}_{n}$,
define the \emph{monomial symmetric function} as 
\begin{equation}
  M_{\lambda}(y_{1},\ldots,y_{m}) =
  \sum_{\genfrac{}{}{0pt}{}{i_{1},\ldots,i_{k}}{\text{distinct terms}}}y_{i_{1}}^{\lambda_{1}}\cdots y_{i_{k}}^{\lambda_{k}}=y_{1}^{\lambda_{1}}\cdots y_{k}^{\lambda_{k}}+\text{symmetric terms}.\label{eq:MZonal}
\end{equation}
\end{defn}
\begin{rem}
An explicit expression of $M_{\lambda}(Y)$ is given by
\cite[eq.~6]{Takemura}:
\begin{equation}\label{eq:MZonalComputation}
  M_{(1^{m_{1}}2^{m_{2}}\cdots)}(Y)=\left(\prod_{j=1}^{n}\frac{1}{m_{j}!}\right)
  \sum_{i_{1},\ldots,i_{k}}y_{i_{1}}^{\lambda_{1}}\cdots y_{i_{k}}^{\lambda_{k}},
\end{equation}
where $\lambda=(\lambda_1,\ldots,\lambda_k)=(1^{m_1},2^{m_2},\ldots,n^{m_n})\in\mathcal{P}_n$ contains $m_1$ $1$'s, $m_2$ $2$'s, etc. 
\end{rem}
\begin{thm}
We have,  for some constants $c_{\kappa,\lambda}$, that (see, e.g., \cite[eq.~13]{Muirhead})
\begin{equation}\label{eq:CInTermsOfM}
  \mathcal{C}_{\kappa}(Y)=\sum_{\lambda\leq\kappa}c_{\kappa,\lambda}\,M_{\lambda}(Y).
\end{equation}
\end{thm}
\begin{rem}\label{rem:polyzero}
  Note that $M_\lambda(Y)$ is defined to be zero whenever there are fewer
  variables than parts in the partition~$\lambda$. It follows that
  $\mathcal{C}_{\kappa}(Y)=0$ if the dimension of~$Y$ is less than the
  number of parts of~$\kappa$.
\end{rem}
\begin{example}\label{Tables}
The following table, from \cite[p.~238]{Muirhead}, shows
the coefficients $c_{\kappa,\lambda}$, in the case $n=4$.
\[
\def\arraystretch{1.3}
\begin{array}{c@{\;\;}|@{\;\;}c@{\quad}c@{\quad}c@{\quad}c@{\quad}c}
\kappa\backslash\lambda & \left(4\right) & \left(3,1\right) & \left(2,2\right) & \left(2,1,1\right) & \left(1,1,1,1\right)\\ \hline
\left(4\right) & 1 & \frac{4}{7} & \frac{18}{35} & \frac{12}{35} & \frac{8}{35}\\
\left(3,1\right) & 0 & \frac{24}{7} & \frac{16}{7} & \frac{88}{21} & \frac{32}{7}\\
\left(2,2\right) & 0 & 0 & \frac{16}{5} & \frac{32}{15} & \frac{16}{5}\\
\left(2,1,1\right) & 0 & 0 & 0 & \frac{16}{3} & \frac{64}{5}\\
\left(1,1,1,1\right) & 0 & 0 & 0 & 0 & \frac{16}{5}
\end{array}
\]
\end{example}
\begin{thm}
The constant $c_{\kappa,\lambda}$ satisfies the recurrence \cite[eq.~14]{Muirhead}
\begin{equation}\label{eq:CRec}
  c_{\kappa,\lambda} =
  \sum_{\lambda<\mu\leq\kappa}\frac{\left(\lambda_{r}+t\right)-\left(\lambda_{s}-t\right)}{\rho_{\kappa}-\rho_{\lambda}}c_{\kappa,\mu},
\end{equation}
where the sum is over all $ \mu=(\lambda_1,\ldots,\lambda_{r-1},\lambda_r+t,\lambda_{r+1},\ldots,
  \lambda_{s-1},\lambda_s-t,\lambda_{s+1}, \allowbreak \ldots,\lambda_k)$, for $\lambda=\left(\lambda_1,\ldots,\lambda_k\right)$
and $t=1,\ldots,\lambda_s$ such that by rearranging the tuple~$\mu$
in a descending order, it lies as $\lambda<\mu\leq\kappa$. Recall the quantity
$\rho_\kappa$ that is defined in~\eqref{eq:RHO}.
\end{thm}

\begin{rem}
As mentioned in Remark 1 on page 73 of  \cite{Takemura}, recurrence \eqref{eq:CRec} fails when $\rho_\kappa-\rho_\lambda=0$, which first occurs when $\kappa=(4,1,1)$ and $\lambda=(3,3)$: 
\[
\rho_\kappa=4\cdot(4-1)+1\cdot(1-2)+1\cdot(1-3)=9=3\cdot(3-1)+3\cdot(3-2)=\rho_\lambda.
\]
In this case, not only does the denominator in \eqref{eq:CRec} vanish, but also the summation is empty. 
It seems that James \cite{James1} has claimed that for all ``relevant'' pairs $(\kappa,\lambda)$, $c_{\kappa,\lambda}>0$. Due to the nonnegative numerator of the summation in \eqref{eq:CRec}, it suggests that if $\rho_\kappa\leq\rho_\lambda$,  $c_{\kappa,\lambda}=0$. 

See the example at the end of Section \ref{sec:Package} of the case $\kappa=(4,1,1)$ and $\lambda=(3,3)$, which is also compatible with the built-in package in SageMath.
\end{rem}

Once the initial value $c_{\kappa,\kappa}$ is given,
\eqref{eq:CRec} can compute $c_{\kappa,\lambda}$ for all $\lambda<\kappa$.
Now, observing the tables in Example~\ref{Tables} and recalling \eqref{eq:TrZonal},
it is easy to see that the sum of each column is given by a multinomial
coefficient. More precisely, let
$\lambda=(\lambda_1,\ldots,\lambda_k)\in\mathcal{P}_{n}$,
\begin{equation}
  \sum_{\kappa=\lambda}^{(n)}c_{\kappa,\lambda} =
  \binom{n}{\lambda_1,\ldots,\lambda_k}.\label{eq:CInitial}
\end{equation}
In particular, $c_{(n),(n)}=\binom{n}{n}=1$.
Thus, all constants $c_{\kappa,\lambda}$ are obtained, and so is $\mathcal{C}_{\lambda}(Y)$.

\section{Characterization of vanishing coefficients}\label{sec:zeros}

Muirhead~\cite[Lem.~7.2.3]{Muirhead} gives a necessary condition for some
coefficients $c_{\kappa,\lambda}$ to be zero, but without a proof. We recall his
result here and give a simple proof of it.
\begin{lemma}\label{lem:zero1}
  Let $n\in\mathbb{N}$ and $\kappa,\lambda\in\mathcal{P}_n$. If $\kappa$
  has more parts than~$\lambda$, then $c_{\kappa,\lambda}=0$.
\end{lemma}
\begin{proof}
  Given a partition $\lambda$, denote by $\operatorname{len}(\lambda)$ its
  number of parts. Assume that $\operatorname{len}(\kappa)=k$. Then, the
  definition of zonal polynomials \eqref{eq:CInTermsOfM} and
  Remark~\ref{rem:polyzero}, we see that
\begin{align*}
 0 &= \mathcal{C}_{\kappa}(y_1,\dots,y_{k-1})\\
    &= 
  \!\!\!\!\!\sum_{\textstyle\genfrac{}{}{0pt}{}{\lambda\leq\kappa}{\operatorname{len}(\lambda)<k}}\!\!\!\!\!
      c_{\kappa,\lambda} \underbrace{M_{\lambda}(y_1,\dots,y_{k-1})}_{\neq0} \quad +
   \sum_{\textstyle\genfrac{}{}{0pt}{}{\lambda\leq\kappa}{\operatorname{len}(\lambda)\geq k}}\!\!\!\!\!
    c_{\kappa,\lambda} \underbrace{M_{\lambda}(y_1,\dots,y_{k-1})}_{=0}.
\end{align*}
  Now, it is apparent that all coefficients $c_{\kappa,\lambda}$
  for which $\lambda$ has fewer parts than~$k$ must be zero.
\end{proof}

While Lemma~\ref{lem:zero1} only gives a necessary condition under which
$c_{\kappa,\lambda}$ is zero, we would like to obtain a full characterization,
i.e., a necessary and sufficient condition for $c_{\kappa,\lambda}=0$. For
example, we have $c_{(8,2,2),(7,4,1)} = 0$, although both partitions have the
same length.
\begin{thm}\label{thm:zero2}
  Let $n\in\mathbb{N}$ and let $\kappa, \lambda\in\mathcal{P}_n$  with $\kappa\geq\lambda$ in lexicographic order. Then $c_{\kappa,\lambda}=0$
  if and only if there exists $p\in\mathbb{N}$ such that,  (where the partitions are filled with zeros as necessary,)
  \begin{equation}\label{eq:prop}
    \sum_{i=1}^p (\kappa_i - \lambda_i) < 0.
  \end{equation}
\end{thm}

\begin{proof}
  We first want to show that, under the assumption~\eqref{eq:prop},
  $\sum_{i=1}^p(\kappa_i-\mu_i)<0$ holds for all $\mu$ of the form
  $\mu=(\dots,\lambda_r+t,\dots,\lambda_s-t,\dots)$, $1\leq t\leq \lambda_s$,
  after reordering the parts. Assume that this reordering requires us
  to move the part $\lambda_r+t$ to the $k$-th position ($k\leq r$) and
  the part $\lambda_s-t$ to the $\ell$-th position ($\ell\geq s$). Then
  the partition~$\mu$ has the following form:
  \[
    (\lambda_1,\dots,\lambda_{k-1},\lambda_r+t,\lambda_k,\dots,
    \lambda_{r-1},\lambda_{r+1},\dots,\lambda_{s-1},\lambda_{s+1},\dots,
    \lambda_{\ell},\lambda_s-t,\lambda_{\ell+1},\dots).
  \]
  Then we have
  \[
    \sum_{i=1}^p \mu_i =
    \begin{cases}
      \sum_{i=1}^p \lambda_i, & \text{if } p<k; \\
      \sum_{i=1}^{p-1} \lambda_i + \lambda_r+t, & \text{if } k\leq p<r; \\
      \sum_{i=1}^p \lambda_i + t, & \text{if } r\leq p<s; \\
      \sum_{i=1}^{s-1} \lambda_i + t + \sum_{i=s+1}^{p+1} \lambda_i, & \text{if } s\leq p<\ell; \\
      \sum_{i=1}^p \lambda_i, & \text{if } p\geq\ell.
    \end{cases}
  \]
  In each of these cases, we claim $\sum_{i=1}^p \mu_i \geq \sum_{i=1}^p
  \lambda_i$.  For the $1$st, $3$rd and $5$th case, it is immediately
  obvious. For the $2$nd and $4$th case, it is true because
  $\lambda_r+t\geq\lambda_p$ and $\lambda_{p+1}\geq\lambda_s-t$, respectively.

  We have just shown that if $\kappa$ and $\lambda$ satisfy~\eqref{eq:prop},
  then all the $\mu$'s in~\eqref{eq:CRec} also satisfy~\eqref{eq:prop}, in
  other words, $c_{\kappa,\lambda}$ is a linear combination of
  $c_{\kappa,\mu}$'s with the same property. In order to conclude the proof by
  induction, we have to investigate the two possible base cases:
  \begin{enumerate}
  \item We arrive at a $c_{\kappa,\lambda}$ for which $\lambda$ has fewer
    parts than~$\kappa$ (, note that the operation
    $(\dots,\lambda_r+t,\dots,\lambda_s-t,\dots)$ weakly decreases the number
    of parts). Then by Lemma~\ref{lem:zero1}, $c_{\kappa,\lambda}=0$.
  \item We arrive at a $c_{\kappa,\lambda}$ such that no suitable $\mu$
    between $\lambda$ and $\kappa$ exists.  Also in this case we get
    $c_{\kappa,\lambda}=0$ since the sum in~\eqref{eq:CRec} is empty.
  \end{enumerate}

  Now we show the converse: if $\kappa,\lambda$ do not satisfy
  \eqref{eq:prop}, then by \eqref{eq:CRec}, $c_{\kappa,\lambda}$ is a linear
  combination (with nonnegative coefficients) of $c$'s that do not satisfy
  \eqref{eq:prop} either.  Note that $\rho_\kappa<\rho_\lambda$ only occurs if
  \eqref{eq:prop} holds, which can be proven as follows. Let $\kappa=(\kappa_{1},\ldots,\kappa_{m})$ and $\lambda=(\lambda_{1},\ldots,\lambda_{m})$, with possible zero components as necessary. For any constant $c$, 
\begin{align*}
\rho_{\kappa}-\rho_{\lambda} & =\sum_{i=1}^{m}\left(\kappa_{i}(\kappa_{i}-i)-\lambda_{i}(\lambda_{i}-i)\right)\\
 & =\sum_{i=1}^{m}(\kappa_{i}-\lambda_{i})(\kappa_{i}+\lambda_{i}-i)\\
 & =\sum_{i=1}^{m}(\kappa_{i}-\lambda_{i})(\kappa_{i}+\lambda_{i}-i+c)-c\sum_{i=1}^{m}(\kappa_{i}-\lambda_{i})\\
 & = \sum_{i=1}^{m}(\kappa_{i}-\lambda_{i})(\kappa_{i}+\lambda_{i}-i+c),
\end{align*}
due to that 
\[
\sum_{i=1}^{m}(\kappa_{i}-\lambda_{i})=\sum_{i=1}^{m}\kappa_{i}-\sum_{i=1}^{k}\lambda_{i}=n-n=0.
\]
Let $c=m+1$ and define $a_i=\kappa_{i}+\lambda_{i}+(m+1-i)$ for $i=1,2,\ldots,m$, which are positive and strictly decreasing. If $\rho_\kappa<\rho_\lambda$ but \eqref{eq:prop}  fails, i.e., for any $p\in\mathbb{N}$, $\sum_{i=1}^p(\kappa_i-\lambda_i)\geq 0$, we see the contradiction as follows.
\begin{align*}
\rho_{\kappa}-\rho_{\lambda} &= \sum_{i=1}^{m}(\kappa_{i}-\lambda_{i})a_{i}\\
  &= a_{m}\sum_{i=1}^{m}(\kappa_{i}-\lambda_{i})+(a_{m-1}-a_{m})\sum_{i=1}^{m-1}(\kappa_{i}-\lambda_{i})+\cdots\\
  &\qquad +(a_{1}-a_{2})(\kappa_{1}-\lambda_{1})\geq 0.
\end{align*}
Thus, we only focus on the case that $\rho_\kappa>\rho_\lambda$. It is not difficult to see that $\lambda$ can be
  converted into~$\kappa$ by a finite sequence of moves
  $(\dots,\lambda_r+1,\lambda_{r+1}-1,\dots)$, where $r$ is the smallest index
  such that $\sum_{i=1}^r (\kappa_i - \lambda_i) > 0$, and all partitions in
  this sequence do not satisfy~\eqref{eq:prop}. The recursion ends when there
  is no such~$r$, i.e., when $\lambda=\kappa$, and one sees that the
  contribution of $c_{\kappa,\kappa}\neq0$ makes $c_{\kappa,\lambda}$ positive.
\end{proof}

\begin{example}
  For $\kappa=(8,2,2)$ and $\lambda=(7,4,1)$ we verify that the second
  partial sum of their differences, i.e., $p=2$ is negative:
  $(8-7)+(2-4)=-1$, and hence $c_{\kappa,\lambda}$ must be zero.
\end{example}

A graphical presentation of the positions where $c_{\kappa,\lambda}=0$ is
shown in Figure~\ref{fig:zeros}. We recognize a fractal-like structure which
is due to Theorem~\ref{thm:zero2} and the lexicographical ordering of
partitions.

\begin{figure}
  \begin{center}
    \includegraphics[width=0.56\textwidth]{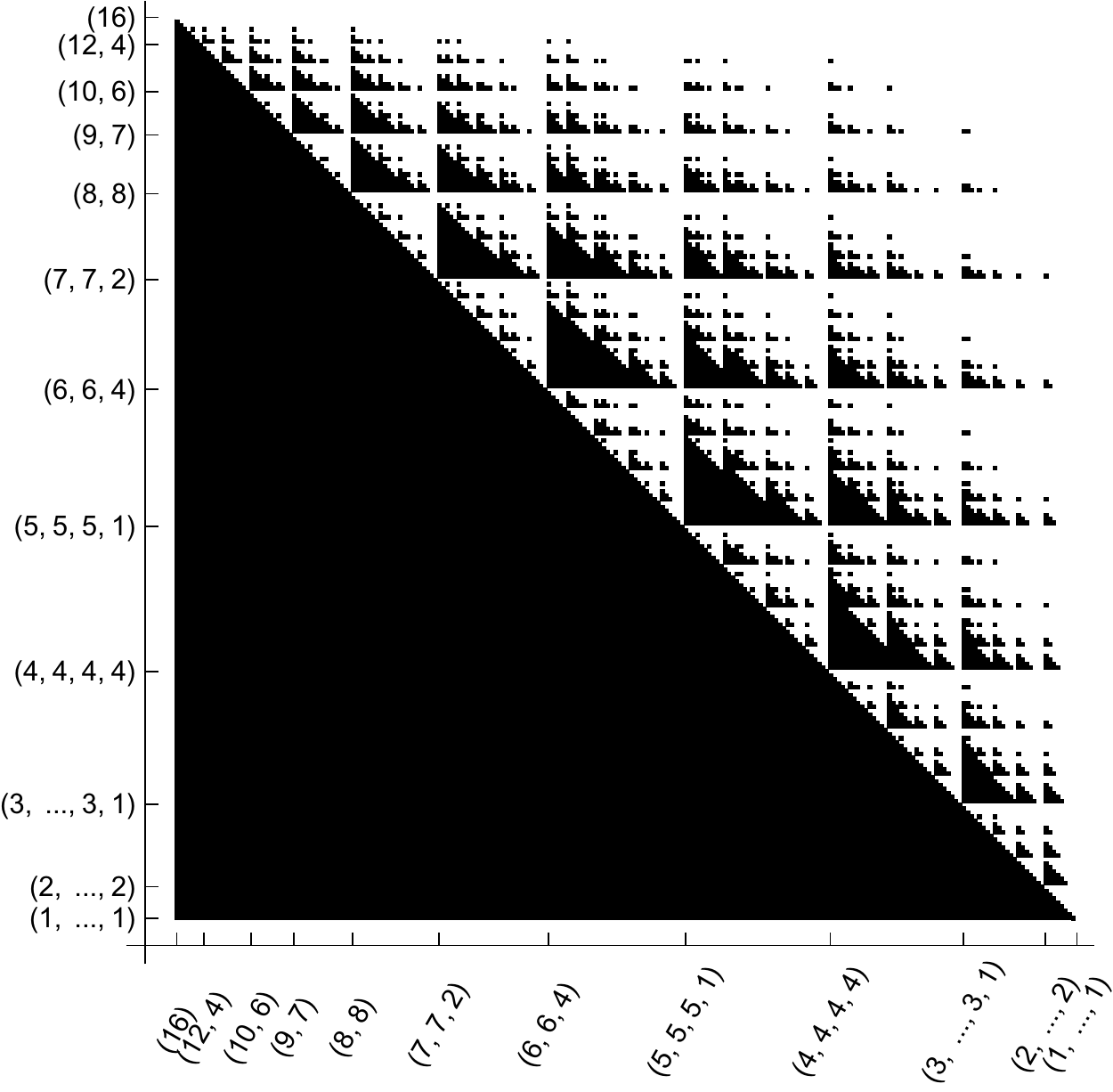}
  \end{center}
  \caption{Location of zeros in the $c_{\kappa,\lambda}$-matrix for $n=16$:
  each nonzero coefficient is represented by a white square, each zero
  by a black square.}
  \label{fig:zeros}
\end{figure}

\section{Infinite families of coefficients $c_{\kappa,\lambda}$}
\label{sec:families}

In this section, we study some infinite families among the coefficients
$c_{\kappa,\lambda}$ of the zonal polynomials~$\mathcal{C}_{\kappa}(Y)$. Since we have
seen that $c_{(n),(n)}=1$ for all~$n$ (that corresponds to the upper left
corner of the $c_{\kappa,\lambda}$-matrix), one can ask whether the
``neighboring'' entries also admit a closed form for general~$n$.
Theorem~\ref{thm:fam1} will give an explicit answer for the first few cases.
Before that, we focus on the first row of the $c_{\kappa,\lambda}$-matrix,
i.e., on the coefficients $c_{\kappa,\lambda}$ for which
$\kappa=(n)$ and $\lambda\leq\kappa$ has two parts.

\begin{thm}\label{thm:row1}
Let $(n-m,m)\in\mathcal{P}_n$, then
\begin{equation}\label{eq:1to2}
  c_{(n),(n-m,m)} =
  \binom{n}{m}\cdot\frac{\bigl(\frac12\bigr)_m}{\bigl(n-m+\frac12\bigr)_m}.
\end{equation}
\end{thm}
\begin{proof}
Obviously, 
$ \rho_{(n)}-\rho_{(n-m,m)}=m(2n-2m+1)$. By \eqref{eq:CRec},
\begin{align*}
  c_{(n),(n-m,m)}
  &= \sum_{(n-m,m)<\mu\leq(n)}\frac{(n-m+t)-(m-t)}
    {\rho_{(n)}-\rho_{(n-m,m)}}\cdot c_{(n),\mu} \\
  &= \frac{1}{m(2n-2m+1)}\sum_{t=1}^{m}(n-2m+2t)c_{(n),\mu}.
\end{align*}
Now, we proceed by induction on~$m$. When $m=1$, the only term in the sum is $t=1$, i.e., $\mu=(n)$,  so that
\begin{equation}\label{eq:cnton1}
  c_{(n),(n-1,1)}=\frac{1}{2n-1}\cdot(n-2+2)\cdot1=\frac{n}{2n-1}=\binom{n}{1}\frac{\frac{1}{2}}{n-\frac{1}{2}}.
\end{equation}
Now, assume \eqref{eq:1to2} holds for $m$. Considering the case $m+1$, we have
\begin{align*}
  & c_{(n),(n-m-1,m+1)} = \\
  &= \frac{1}{(m+1)(2(n-m)-1)}\sum_{t=1}^{m+1}(n-2m-2+2t)c_{(n),\mu} \\
  &= \frac{1}{(m+1)2(n-m)-1)}\biggl((n-2m)c_{(n),(n-m,m)}+\sum_{t=1}^{m}(n-2m+2t)c_{(n),\mu}\biggr) \\
  &= \frac{c_{(n),(n-m,m)}}{(m+1)(2(n-m)-1)}\Bigl((n-2m)+m(2(n-m)+1)\Bigr) \\
  &= \frac{(n-m)(2m+1)}{(m+1)(2(n-m)-1)}\cdot\binom{n}{m}\frac{\bigl(\frac12\bigr)_m}{\bigl(n-m+\frac12\bigr)_m}\allowdisplaybreaks \\
  &= \binom{n}{m+1}\cdot\frac{\bigl(\frac12\bigr)_{m+1}}{\bigl(n-m-\frac12\bigr)_{m+1}}.\qedhere
\end{align*}
\end{proof}

\begin{thm}
\label{thm:fam1}
We have
\begin{alignat*}{2}
  c_{(n-1,1),(n-1,1)} &= \frac{2n(n-1)}{2n-1}, &\qquad& (n\geq2);\\
  c_{(n-1,1),(n-2,2)} &= \frac{2n(n-1)(n-2)}{(2n-1)(2n-5)}, && (n\geq4); \\
  c_{(n-2,2),(n-2,2)} &= \frac{2n(n-1)(n-2)(n-3)}{(2n-3)(2n-5)}, &&(n\geq4).
\end{alignat*}
\end{thm}
\begin{proof}
Using \eqref{eq:cnton1} and  \eqref{eq:CInitial} with $\lambda=(n-1,1)$, we have
\begin{align*}
 c_{(n-1,1),(n-1,1)}&=\binom{n}{n-1,1}-c_{(n),(n-1,1)}\\
&=n-\frac{n}{2n-1}=\frac{2n^{2}-2n}{2n-1}=\frac{2n\left(n-1\right)}{2n-1}.
\end{align*}
By applying \eqref{eq:CRec}, we obtain
\begin{align*}
  c_{(n-1,1),(n-2,2)}
  &= \frac{(n-2+1)-(2-1)}{\rho_{(n-1,1)}-\rho_{(n-2,2)}}\cdot c_{(n-1,1),(n-1,1)}\\
  &= \frac{(n-2+1)-(2-1)}{2n-5}\cdot\frac{2n(n-1)}{2n-1}= \frac{2n(n-1)(n-2)}{(2n-1)(2n-5)}.
\end{align*}
Finally, using~\eqref{eq:CInitial} and \eqref{eq:1to2}with $\lambda=(n-2,2)$, we get the last coefficient:
\begin{align*}
 c_{(n-2,2),(n-2,2)}&=\binom{n}{n-2,2}-c_{(n),(n-2,2)}-c_{(n-1,1),(n-2,2)}\\
&=\frac{2n(n-1)(n-2)(n-3)}{(2n-3)(2n-5)}.\qedhere
\end{align*}
\end{proof}

It is clear that we could proceed in this manner and compute more coefficients
$c_{\kappa,\lambda}$ in the upper left corner of the matrix for symbolic~$n$.
Since the corresponding calculations get too tedious to be done by hand, we
employ computer algebra to determine the rational function expressions for a
few more coefficients.  Note that, when $n$ is sufficiently large, the
lexicographically largest elements of $\mathcal{P}_n$ are (in descending
order): 
\[
  (n), \; (n-1,1), \; (n-2,2), \; (n-2,1,1), \; (n-3,3), \; (n-3,2,1), \; \text{etc.}
\]
In Table~\ref{tab:cn1} we give closed forms for the coefficients
$c_{\kappa,\lambda}$ when $\kappa$ and $\lambda$ are taken from these
lexicographically largest partitions, i.e., when both $\kappa$ and $\lambda$
are of the form $(n-m,\pi)$, where $\pi\in\mathcal{P}_m$ 
but $n$ is symbolic.  Pictorially speaking, this table represents the
upper left submatrix of the $c_{\kappa,\lambda}$-matrix for large~$n$; more
precisely, for $n\geq6$ since we consider only partitions
$(n),\dots,(n-3,1,1,1)$. We remark that all these formulas have been
rigorously proven by applying \eqref{eq:CRec} and~\eqref{eq:CInitial} for
symbolic~$n$, as it was done in the proof of Theorem~\ref{thm:fam1}. This
symbolic proof strategy is implemented in our Mathematica package as
\textbf{ZonalCoefficientN} (see Section~\ref{sec:mma}).
\begin{table}
\begin{center}
\parbox{0.85\textwidth}{
$
 \begin{array}{c@{\;\;}|@{\;\;}c@{\quad}c@{\quad}c@{\quad}c}
   \kappa\backslash\lambda & (n) & (n-1,1) & (n-2,2) & (n-2,1,1)\\
   \hline \rule{0pt}{14pt}
   (n) & 1 & \frac{n}{2n-1} & \frac{3(n-1)n}{2(2n-3)(2n-1)} & \frac{(n-1)n}{(2n-3)(2n-1)}\\[1ex]
   (n-1,1) & 0 & \frac{2(n-1)n}{2n-1} & \frac{2(n-2)(n-1)n}{(2n-5)(2n-1)} & \frac{2n(2n^{2}-6n+3)}{(2n-5)(2n-1)}\\[1ex]
   (n-2,2) & 0 & 0 & \frac{2(n-3)(n-2)(n-1)n}{(2n-5)(2n-3)} & \frac{4(n-3)(n-2)(n-1)n}{3(2n-5)(2n-3)}\\[1ex]
   (n-2,1,1) & 0 & 0 & 0 & \frac{2}{3}(n-2)n
 \end{array}
$ \\[1ex]
$
 \begin{array}{c@{\;\;}|@{\;\;}c@{\quad}c}
   \kappa\backslash\lambda & (n-3,3) & (n-3,2,1)\\
   \hline \rule{0pt}{14pt}
   (n) & \frac{5(n-2)(n-1)n}{2(2n-5)(2n-3)(2n-1)} & \frac{3(n-2)(n-1)n}{2(2n-5)(2n-3)(2n-1)}\\[1ex]
   (n-1,1) & \frac{3(n-3)(n-2)(n-1)n}{(2n-7)(2n-5)(2n-1)} & \frac{(n-2)n(5n^{2}-20n+11)}{(2n-7)(2n-5)(2n-1)}\\[1ex]
   (n-2,2) & \frac{2(n-4)(n-3)(n-2)(n-1)n}{(2n-9)(2n-5)(2n-3)} & \frac{2(n-3)(n-1)n(5n^{2}-30n+36)}{3(2n-9)(2n-5)(2n-3)}\\[1ex]
   (n-2,1,1) & 0 & \frac{2(n-3)(n-2)n}{3(2n-7)}\\[1ex]
   (n-3,3) & \frac{4(n-5)(n-4)(n-3)(n-2)(n-1)n}{3(2n-9)(2n-7)(2n-5)} & \frac{4(n-5)(n-4)(n-3)(n-2)(n-1)n}{5(2n-9)(2n-7)(2n-5)}\\[1ex]
   (n-3,2,1) & 0 & \frac{4(n-4)(n-3)(n-1)n}{5(2n-7)}
 \end{array}
$
}\\[3ex]
\end{center}
\caption{Coefficients $c_{\kappa,\lambda}$ for some of the lexicographically
largest partitions of~$n$; the lower table continues the upper one
to the right.}
\label{tab:cn1} 
\end{table}
\begin{example}
For example, for $n=23$ we obtain the coefficient $c_{(21,2),(21,1,1)}$ by
reading the entry in Table~\ref{tab:cn1} in row $(n-2,2)$ and column
$(n-2,1,1)$:
\[
  c_{(21,2),(21,1,1)} = \frac{4 (n-3) (n-2) (n-1) n}{3 (2 n-5) (2 n-3)} \Bigg|_{n=23}
  = \frac{4 \cdot 20 \cdot 21 \cdot 22 \cdot 23}{3 \cdot 41 \cdot 43}
  = \frac{283360}{1763}.
\]
\end{example}

We formulate a conjecture on the limit, as $n$ goes to infinity,
of the coefficients of the form $c_{(n),(n-n',\lambda')}$.

\begin{conj}
Let $n'<n$ and $\lambda'=(p_1,\ldots,p_\ell)\in\mathcal{P}_{n'}$. Then, 
\[
\lim_{n\to\infty} c_{(n),(n-n',\lambda') }=\prod_{i=1}^{\ell} \frac{(p_i)_{p_i}}{p_i! \, 2^{2p_i-1}}
\]
\end{conj}

\begin{example}
Let $n'=3$ and $\lambda'=(2,1)$. By Table \ref{tab:cn1} and the conjecture above
\begin{align*}
  \lim_{n\rightarrow\infty}c_{(n),(n-3,2,1)} &=
  \lim_{n\rightarrow\infty}\frac{3(n-2)(n-1)n}{2(2n-5)(2n-3)(2n-1)}=\frac{3}{16} \\
  &= \frac{(2)_2}{2! \, 2^{2\cdot 2-1}}\cdot\frac{(1)_1}{1! \, 2^{2\cdot 1-1}}.
\end{align*}
\end{example}

The next families we study are located in the ``lower right corner'' of
$c_{\kappa,\lambda}$-matrix. In this case both $\kappa$ and $\lambda$ are of
the form $(2^m,1^{n-2m})$, i.e., a sequence of $m$ times the part~$2$ and
$n-2m$ times the part~$1$. Note that for $m=0,1,2,\dots$ we obtain the
lexicographically smallest partitions of~$n$ (in increasing order).  In
contrast to Theorem~\ref{thm:fam1} it is not straightforward to derive
formulas for general~$n$ by means of Equations \eqref{eq:CRec}
and~\eqref{eq:CInitial}. Instead, we compute the first few values of these
sequences, say up to $n=30$, and then guess a closed form expression.
Some results are displayed in Table~\ref{tab:cn2}.

More precisely, we first applied the \texttt{Guess.m} package~\cite{Guess} to
find a plausible candidate for a linear recurrence with polynomial
coefficients (in all considered instances this recurrence was of first order,
which easily allowed for a closed form solution). Then, after observing that
all expressions obtained this way were of the form $2^n\cdot r(n)$ where
$r(n)$ is some rational function in~$n$, we refined our ansatz to only search
for expressions of this form: divide the $n$-th sequence entry by $2^n$ and
then perform polynomial interpolation and rational reconstruction.
\begin{table}
\begin{center}
\parbox{0.85\textwidth}{
$
  \begin{array}{c@{\;\;}|@{\;\;}c@{\quad}c@{\quad}c}
    \kappa\backslash\lambda
    & (2^4,1^{n-8}) & (2^3,1^{n-6}) & (2^2,1^{n-4})
    \\ \hline
    \rule{0pt}{14pt}
    (2^4,1^{n-8})
    & \frac{2^{n-3} (n-6) (n-5) n}{15}
    & \frac{2^{n-3} (n-7) (n-6)^2 n}{15}
    & \frac{2^{n-4} (n-7) (n-6)^2 (n-5) n}{15}
    \\[1ex]
    (2^3,1^{n-6})
    & 0
    & \frac{2^{n-3} (n-4) (n-3)}{3}
    & \frac{2^{n-3} (n-5) (n-4)^2}{3}
    \\[1ex]
    (2^2,1^{n-4})
    & 0 & 0
    & \frac{2^{n-1} (n-2) (n-1)}{3 (n+1)}
  \end{array}
$ \\[1ex]
$
  \begin{array}{c@{\;\;}|@{\;\;}c@{\quad}c}
    \kappa\backslash\lambda
    & (2,1^{n-2}) & (1^n) \\ \hline
    \rule{0pt}{14pt}
    (2^4,1^{n-8})
    & \frac{2^{n-4} (n-7) (n-6)^2 (n-5) (n-2) n}{45}
    & \frac{2^{n-6} (n-7) (n-6)^2 (n-5) (n-1) n^2}{45}
    \\[1ex]
    (2^3,1^{n-6})
    & \frac{2^{n-4} (n-5) (n-4)^2 (n-3)}{3}
    & \frac{2^{n-4} (n-5) (n-4)^2 (n-3) n}{9}
    \\[1ex]
    (2^2,1^{n-4})
    & \frac{2^{n-1} (n-3) (n-2)^2}{3 (n+1)}
    & \frac{2^{n-2} (n-3) (n-2)^2 (n-1)}{3 (n+1)}
    \\[1ex]
    (2,1^{n-2})
    & \frac{2^{n-1} n}{n+2}
    & \frac{2^{n-1} (n-1) n^2}{(n+1) (n+2)}
    \\[1ex]
    (1^n)
    & 0 & \frac{2^n}{n+1}
  \end{array}
$
}\\[3ex]
\end{center}
\caption{Coefficients $c_{\kappa,\lambda}$ for some of the lexicographically
  smallest partitions of~$n$; the lower table continues the upper one to the
  right.}
\label{tab:cn2}
\end{table}

\section{Partitions with two parts}\label{sec:2parts}

Now we concentrate on the coefficients $c_{\kappa,\lambda}$ when both $\kappa$
and $\lambda$ have at most two parts. Takemura~\cite[\S~4.4]{Takemura}
gives formulas for these coefficients in terms of elementary symmetric
functions, while we focus on monomial symmetric functions.

When applying the recursive formula~\eqref{eq:CRec} one sees that the
partitions~$\mu$ cannot have more parts than~$\lambda$, by the way how they
are constructed. Similarly, when we use \eqref{eq:CInitial} to compute
$c_{\lambda,\lambda}$, only those coefficients $c_{\kappa,\lambda}$ contribute
for which $\kappa$ has at most two parts; this is a direct consequence of
Lemma~\ref{lem:zero1}.  We conclude that, for the computation of
$c_{\kappa,\lambda}$, we do not need any $c_{\kappa',\lambda'}$ with $\kappa'$
or $\lambda'$ having more than two parts.  Throughout this section we write
\[
  \kappa=(a,a-b) \quad\text{and}\quad \lambda=(a-d,a-b+d)
\]
for nonnegative integers $a,b,d$.  The condition $b<a$ ensures that $\kappa$
is a proper partition and the condition $d\leq b/2$ ensures that the
parts of $\lambda$ are in the correct order. In Theorem~\ref{thm:cf2} we
present a closed form of the coefficients $c_{\kappa,\lambda}$ under these
assumptions. We start our investigation by specializing
 \eqref{eq:CRec} and \eqref{eq:CInitial} to partitions with two
parts. First of all, it is easy to compute that 
\[
  \rho_{\kappa}-\rho_{\lambda} = d\cdot(2b-2d+1).
\]
Then, recurrence~\eqref{eq:CRec} can be written as, by noting that
$\lambda$ has only two parts,
\begin{align}\label{eq:rec2}
  c_{\kappa,\lambda} = c_{(a,a-b),(a-d,a-b+d)} &=
  \sum_{\lambda<\mu\leq\kappa}\frac{(a-d+t)-(a-b+d-t)}{\rho_{\kappa}-\rho_{\lambda}} c_{\kappa,\mu} \notag \\
  &= \sum_{t=1}^d\frac{b-2d+2t}{d\,(2b-2d+1)}c_{(a,a-b),(a-d+t,a-b+d-t)} \notag \\ 
  & = \sum_{j=0}^{d-1} \frac{b-2j}{d\,(2b-2d+1)}c_{(a,a-b),(a-j,a-b+j)}. \allowdisplaybreaks
\end{align}

\begin{example}\label{ex:twoparts}
We consider some specific examples.  For $d=1,2,3$, recursively applying \eqref{eq:rec2} yields
\allowdisplaybreaks
\begin{align*}
  c_{(a,a-b),(a-1,a-b+1)} &=
    \frac{b}{2b-1}\,c_{(a,a-b),(a,a-b)}, \\
  c_{(a,a-b),(a-2,a-b+2)} &=
\frac{3b(b-1)}{2(2b-1)(2b-3)}\,c_{(a,a-b),(a,a-b)}, \\
  c_{(a,a-b),(a-3,a-b+3)} &=
\frac{5b(b-1)(b-2)}{2(2b-1)(2b-3)(2b-5)}\,c_{(a,a-b),(a,a-b)}.
\end{align*}
\end{example}

\begin{prop}\label{prop:p2row}
  Given two partitions $\kappa=(a,a-b)$ and $\lambda=(a-d,a-b+d)$ of the positive integer $n=2a-b$,
  with $0\leq b<a$ and $0\leq d\leq b/2$, then
  \begin{equation}\label{eq:2PartRec}
    c_{\kappa,\lambda} = c_{(a,a-b),(a-d,a-b+d)} =
    \binom{b}{d}\,\frac{\bigl(\frac12\bigr)_{\!d}}{\bigl(b-d+\frac12\bigr)_{\!d}} \cdot c_{(a,a-b),(a,a-b)}.
  \end{equation}
\end{prop}
\begin{proof}
  Note that this result is a direct generalization of Theorem~\ref{thm:row1},
  and hence the same inductive argument could be applied. Instead, we
  illustrate a different proof strategy, using symbolic summation. Because of
  the recursive definition of the coefficients $c_{\kappa,\lambda}$,
  demonstrated in Example~\ref{ex:twoparts}, it suffices to show that the
  asserted expression satisfies the recurrence~\eqref{eq:rec2}. Namely, we
  need to prove
  \begin{equation}\label{eq:sumid}
    \binom{b}{d}\,\frac{\bigl(\frac12\bigr)_{\!d}}{\bigl(b-d+\frac12\bigr)_{\!d}} =
    \sum_{j=0}^{d-1} \binom{b}{j}\,\frac{(b-2j)\,\bigl(\frac12\bigr)_{\!j}}{d\,(2b-2d+1)\,\bigl(b-j+\frac12\bigr)_{\!j}}.
  \end{equation}
  Using special-purpose computer algebra packages, such as the
  HolonomicFunctions package~\cite{HolonomicFunctions}, we find that the
  expression in the sum, denote it by $f(j)$, is Gosper-summable. More
  precisely, we find a function
  \[
    g(j) = \frac{j\,(2b-2j+1)}{b-2j} \cdot f(j) =
    \binom{b}{j}\frac{j\,(2b-2j+1) \bigl(\frac12\bigr)_{\!j}}{d\,(2b-2d+1)\,\bigl(b-j+\frac12\bigr)_{\!j}}
  \]
  with the property $g(j+1)-g(j)=f(j)$ (the latter can be easily
  verified). By telescoping, and by noting that $g(0)=0$, we obtain
  the value of the right-hand side of~\eqref{eq:sumid}:
  \[
    g(d) = \binom{b}{d}\cdot\frac{\bigl(\frac12\bigr)_{\!d}}{\bigl(b-d+\frac12\bigr)_{\!d}},
  \]
  which matches exactly the left-hand side of~\eqref{eq:sumid}.
\end{proof}

\begin{thm}\label{thm:cf2}
  Let $a,b,d\in\mathbb{N}$ with $0\leq b<a$ and $0\leq d\leq b/2$.
  Then we have
  \[
    c_{(a,a-b),(a-d,a-b+d)} = 
    \frac{(2a-b)! \, \bigl(b+\frac12\bigr) \, \bigl(\frac12\bigr)_{\!d}}
         {d! \, (a-b)! \, (b-d)! \, \bigl(b-d+\frac12\bigr)_{\!a-b+d+1}}.
  \]
\end{thm}
\begin{proof}
  We first show that the asserted expression is compatible with the result of
  Proposition~\ref{prop:p2row}: indeed, by computing the quotient
  \begin{align*}
    \frac{c_{(a,a-b),(a-d,a-b+d)}}{c_{(a,a-b),(a,a-b)}} &=
    \frac{(2a-b)! \, \bigl(b+\frac12\bigr) \, \bigl(\frac12\bigr)_{\!d} \cdot
      (a-b)! \, b! \, \bigl(b+\frac12\bigr)_{\!a-b+1}}
      {d! \, (a-b)! \, (b-d)! \, \bigl(b-d+\frac12\bigr)_{\!a-b+d+1} \!\cdot
        (2a-b)! \, \bigl(b+\frac12\bigr)} \\
    &= \frac{b! \, \bigl(\frac12\bigr)_{\!d} \, \bigl(b+\frac12\bigr)_{\!a-b+1}}
      {d! \, (b-d)! \, \bigl(b-d+\frac12\bigr)_{\!a-b+d+1}}
    = \binom{b}{d}\,\frac{\bigl(\frac12\bigr)_{\!d}}{\bigl(b-d+\frac12\bigr)_{\!d}},
  \end{align*}
  we see that this is the case. It remains to prove that the asserted
  expression is correct in the case $d=0$, i.e., when $\kappa=\lambda$.  For
  this purpose, we employ the recursion~\eqref{eq:CInitial}, specialized to
  partitions with two parts:
  \begin{equation}\label{eq:CInitial2}
    \sum_{d=0}^{a-b} c_{(a+d,a-b-d),(a,a-b)} = \binom{2a-b}{a}.
  \end{equation}
  By dividing both sides with the binomial coefficient of the right-hand side,
  and by inserting the asserted closed form (after the change of variables
  $a\to a+d$ and $b\to b+2d$), we are left with the summation identity
  \begin{equation}\label{eq:sumid2}
    \sum_{d=0}^{a-b}
    \frac{a! \, (a-b)! \, \bigl(b+2d+\frac{1}{2}\bigr) \, \bigl(\frac{1}{2}\bigr)_{\!d}}
         {d! \, (a-b-d)! \, (b+d)! \, \bigl(b+d+\frac12\bigr)_{\!a-b+1}} = 1.
  \end{equation}
  Taking into account the recursive definition of the coefficients
  $c_{\kappa,\lambda}$, the (inductive) proof is completed by
  verifying~\eqref{eq:sumid2}.
  For this purpose, we denote by $f(a,b,d)$ the expression inside the
  sum~\eqref{eq:sumid2} and construct two WZ pairs, i.e., two functions
  \begin{align*}
    g_1(d) &= \frac{-2d(b+d)}{(a-b-d+1)(2b+4d+1)}\cdot f(a,b,d) \\
    &= -\frac{a! \, (a-b)! \, \bigl(\frac12\bigr)_{\!d}}{(d-1)! \, (b+d-1)! \, (a-b-d+1)! \, \bigl(b+d+\frac12\bigr)_{\!a-b+1}}, \\
    g_2(d) &= \frac{d(2a+2d+1)}{(a-b)(2b+4d+1)} \cdot f(a,b,d)\\
  &= \frac{a! \, (a-b-1)! \, \bigl(\frac12\bigr)_{\!d}}{(d-1)! \, (b+d)! \, (a-b-d)! \, \bigl(b+d+\frac12\bigr)_{\!a-b}},
  \end{align*}
  such that the following identities hold (they can be verified by routine
  calculations):
  \begin{align}
    f(a+1,b,d)-f(a,b,d) &= g_1(d+1) - g_1(d), \label{eq:WZ1} \\
    f(a,b+1,d)-f(a,b,d) &= g_2(d+1) - g_2(d). \label{eq:WZ2}
  \end{align}
  Now we sum \eqref{eq:WZ1} for $d=0,\dots,a-b$ and obtain
  \[
    \sum_{d=0}^{a-b}\bigl(f(a+1,b,d)-f(a,b,d)\bigr) = g_1(a-b+1) - g_1(a,0),
  \]
  or equivalently,
  \[
    \sum_{d=0}^{a-b+1}f(a+1,b,d) - \sum_{d=0}^{a-b}f(a,b,d) = g_1(a-b+1) - g_1(a,0) + f(a+1,b,a-b+1).
  \]
  A straightforward calculation shows that the left-hand side equals~$0$,
  thereby showing that the sum $\sum_{d=0}^{a-b}f(a,b,d)$ is independent
  of~$a$. Summing over \eqref{eq:WZ2}, followed by a similar calculation,
  shows that the sum does not depend on~$b$ either. Therefore, the sum
  in~\eqref{eq:sumid2} is constant, and by setting $a=b=0$, one immediately
  sees that this constant is~$1$.
\end{proof}

\begin{rem}
  By setting $b=a$ in Theorem~\ref{thm:cf2} and by interpreting $(a,0)$ as the
  partition~$(a)$, we recover Theorem~\ref{thm:row1}:
  \[
    c_{(a),(a-d,d)} = 
    \frac{(a)! \, \bigl(a+\frac12\bigr) \, \bigl(\frac12\bigr)_{\!d}}
         {d! \, (a-d)! \, \bigl(a-d+\frac12\bigr)_{\!d+1}} =
    \binom{a}{d} \, \frac{\bigl(\frac12\bigr)_{\!d}}{\bigl(a-d+\frac12\bigr)_{\!d}}.
  \]
\end{rem}

\section{Partitions with three and four parts}\label{sec:3+4parts}

We have seen that the coefficients of the zonal polynomial $\mathcal{C}_\kappa(Y)$ are
given by the row indexed by~$\kappa$ in the $c_{\kappa,\lambda}$-matrix. Using
\eqref{eq:CRec} we can express all coefficients~$c_{\kappa,\lambda}$ in the
$\kappa$-th row as constant multiples of the diagonal
coefficient~$c_{\kappa,\kappa}$.  Unfortunately, the latter one is harder to
obtain: to apply \eqref{eq:CInitial} we need to know all $c_{\kappa,\lambda}$
in the $\lambda$-th column, which in turn are obtained by \eqref{eq:CRec} and
so on. Hence, in the worst case, we need to compute the whole triangle above
the position $(\kappa,\kappa)$.

Therefore, it would be highly desirable to have a more direct way to compute
the diagonal coefficients~$c_{\kappa,\kappa}$. We present formulas for the
special cases that $\kappa$ has three resp.\ four parts.

\begin{conj}\label{conj:diag3}
  Let $\kappa=(a,a-b,a-c)$ with integers $0\leq b\leq c\leq a$ be a partition
  of $n=3a-b-c$ into at most three parts. Then the diagonal coefficient
 \[ 
 c_{\kappa,\kappa}=\frac{(c+1)!}{(a+1)!} \cdot \frac{n!}{\delta_1! \, \delta_2! \, \delta_3! \,
     \bigl(\delta_1+\frac32\bigr)_{\!\delta_2} \, \bigl(\delta_2+\frac32\bigr)_{\!\delta_3}}, 
 \] 
  with $\delta_1=\kappa_1-\kappa_2=b$, $\delta_2=\kappa_2-\kappa_3=c-b$, and
  $\delta_3=\kappa_3-\kappa_4=a-c$ being the differences between consecutive
  parts of~$\kappa$ (with the convention $\kappa_4=0$).
\end{conj}

\begin{conj}\label{conj:diag4}
  Let $\kappa=(a,a-b,a-c,a-d)$ with integers $0\leq b\leq c\leq d\leq a$ be a
  partition of $n=4a-b-c-d$ into at most four parts. Then the diagonal
  coefficient $c_{\kappa,\kappa}$ is given by
\[
\frac{(c+1)! \, (d-b+1)!}{(a-b+1)! \, (d+1)! \,  \bigl(d+\frac52\bigr)_{\!a-d}}
    \cdot \frac{n!}{\delta_1! \, \delta_2! \, \delta_3! \, \delta_4! \,
      \bigl(\delta_1+\frac32\bigr)_{\!\delta_2} \bigl(\delta_2+\frac32\bigr)_{\!\delta_3} \bigl(\delta_3+\frac32\bigr)_{\!\delta_4}},  
\]
  with $\delta_1=\kappa_1-\kappa_2=b$, $\delta_2=\kappa_2-\kappa_3=c-b$,
  $\delta_3=\kappa_3-\kappa_4=d-c$, and $\delta_4=\kappa_4-\kappa_5=a-d$ being
  the differences between consecutive parts of~$\kappa$.
\end{conj}

We have verified Conjecture~\ref{conj:diag3} for all $0\leq b\leq c\leq a\leq
14$ and Conjecture~\ref{conj:diag4} for all $0\leq b\leq c\leq d\leq a\leq
10$.

\begin{rem}
Conjecture \ref{conj:diag3} with $c=a$ reduces to  Theorem~\ref{thm:cf2}  with $d=0$. Similarly, by letting $d=a$ in
Conjecture \ref{conj:diag4}, we have Conjecture \ref{conj:diag3}. 
\end{rem}

\section{SageMath package for the calculation of zonal polynomials}\label{sec:Package}

We briefly present the main functionality of the package \texttt{Zonal.sage},
which is freely available at \url{https://jiulin90.github.io/package.html},
and give a few examples. On the same website, we provide a
manual with further examples of its usage.

Let $\kappa, \lambda\in\mathcal{P}_n$ be partitions and $Y=(a, b, c, \ldots )$ be the variables.
\vspace{10bp}

\noindent \textbf{CZonal}$(\lambda, Y)$
computes the zonal polynomial $\mathcal{C}_{\lambda}(Y)$, by \eqref{eq:CInTermsOfM}.

\noindent\fbox{\begin{minipage}[t]{0.98\columnwidth \fboxsep \fboxrule}

\texttt{\textcolor{orange}{sage:}}\texttt{ load(}\texttt{\textcolor{cyan}{\textsf{'}}}\texttt{\textcolor{cyan}{Zonal.sage}}\texttt{\textcolor{cyan}{\textsf{'}}}\texttt{)}

\texttt{\textcolor{orange}{sage:}}\texttt{ var(}\texttt{\textcolor{cyan}{\textsf{'}}}\texttt{\textcolor{cyan}{a}}\texttt{\textcolor{cyan}{\textsf{'}}}\texttt{,}\texttt{\textcolor{cyan}{\textsf{'}}}\texttt{\textcolor{cyan}{b}}\texttt{\textcolor{cyan}{\textsf{'}}}\texttt{,}\texttt{\textcolor{cyan}{\textsf{'}}}\texttt{\textcolor{cyan}{c}}\texttt{\textcolor{cyan}{\textsf{'}}}\texttt{)}

\texttt{(a,b,c)}
\texttt{\textcolor{orange}{sage:}}\texttt{ CZonal({[}}\texttt{\textcolor{cyan}{2}}\texttt{\textcolor{black}{,}}\texttt{\textcolor{cyan}{1}}\texttt{{]},{[}a,b,c{]})}

\texttt{12/5{*}a\textasciicircum{}2{*}b + 12/5{*}a{*}b\textasciicircum{}2
+ 12/5{*}a\textasciicircum{}2{*}c + 18/5{*}a{*}b{*}c + 12/5{*}b\textasciicircum{}2{*}c}

\texttt{+ 12/5{*}a{*}c\textasciicircum{}2 + 12/5{*}b{*}c\textasciicircum{}2 }
\end{minipage}}

\vspace{5bp}

\noindent \textbf{MZonal}$(\lambda,Y)$
computes the monomial symmetric function $M_{\lambda}(Y)$ 
by \eqref{eq:MZonalComputation}.

\noindent\fbox{\begin{minipage}[t]{0.98\columnwidth \fboxsep \fboxrule}

\texttt{\textcolor{orange}{sage:}}\texttt{ MZonal({[}}\texttt{\textcolor{cyan}{2}}\texttt{\textcolor{black}{,}}\texttt{\textcolor{cyan}{2}}\texttt{\textcolor{black}{,}}\texttt{\textcolor{cyan}{1}}\texttt{{]},{[}a,b,c{]})}

\texttt{a\textasciicircum{}2{*}b\textasciicircum{}2{*}c + a\textasciicircum{}2{*}b{*}c\textasciicircum{}2
+ a{*}b\textasciicircum{}2{*}c\textasciicircum{}2}

\texttt{\textcolor{orange}{sage:}}\texttt{ MZonal({[}}\texttt{\textcolor{cyan}{2}}\texttt{\textcolor{black}{,}}\texttt{\textcolor{cyan}{1}}\texttt{{]},{[}a,b,c{]})}

\texttt{a\textasciicircum{}2{*}b + a{*}b\textasciicircum{}2 + a\textasciicircum{}2{*}c
+ b\textasciicircum{}2{*}c + a{*}c\textasciicircum{}2 + b{*}c\textasciicircum{}2}
\end{minipage}}

\vspace{5bp}
\pagebreak[1]

\noindent \textbf{Coeffi}$(\kappa,\lambda)$
computes the coefficient~$c_{\kappa,\lambda}$ for partitions $\kappa\geq\lambda$, by (\ref{eq:CRec}) and (\ref{eq:CInitial}). 

\noindent\fbox{\begin{minipage}[t]{0.98\columnwidth \fboxsep \fboxrule}
\texttt{\textcolor{orange}{sage:}}\texttt{ Coeffi({[}}\texttt{\textcolor{cyan}{5}}\texttt{\textcolor{black}{,}}\texttt{\textcolor{cyan}{4}}\texttt{{]},{[}}\texttt{\textcolor{cyan}{3}}\texttt{,}\texttt{\textcolor{cyan}{3}}\texttt{,}\texttt{\textcolor{cyan}{3}}\texttt{{]})}

\texttt{82944/1925}
\end{minipage}}

\begin{rem}
The SageMath software has built-in functions for Jack symmetric
functions, as mentioned in the Introduction,
where the zonal polynomials
$\mathcal{Z}_{\lambda}(Y):=\mathcal{C}_\lambda(Y)/c_{\lambda,\lambda}$ are also
implemented, as a special case of Jack polynomials. Many properties can be
checked, such as algebraic relations among Jack polynomials in the $P$, $J$, and
$Q$ bases. In particular, an example shows that
\[
  \left(\mathcal{Z}_{(2)}(Y)\right)^2=\frac{64}{45}\mathcal{Z}_{(2, 2)}(Y) +
  \frac{16}{21}\mathcal{Z}_{(3, 1)}(Y) + \mathcal{Z}_{(4)}(Y).
\]
In addition, one can expand $\mathcal{Z}_{\lambda}(Y)$ into an explicit expression, by
the command \textbf{expand}. Meanwhile, since our package only focuses on
calculations of zonal polynomials, it is notably faster by using our
\textbf{CZonal}. The screenshot below shows two computations of
$\mathcal{Z}_{(4,1,1)}(a,b,c)$ and $\mathcal{C}_{(4,1,1)}(a,b,c)$,
which also confirm
Conjecture \ref{conj:diag3} with $c_{(4,1,1)(4,1,1)}=
\mathcal{C}_{(4,1,1)}(a,b,c)/\mathcal{Z}_{(4,1,1)}(a,b,c)=16$.
\begin{center}
\includegraphics[width=\textwidth, height=7cm]{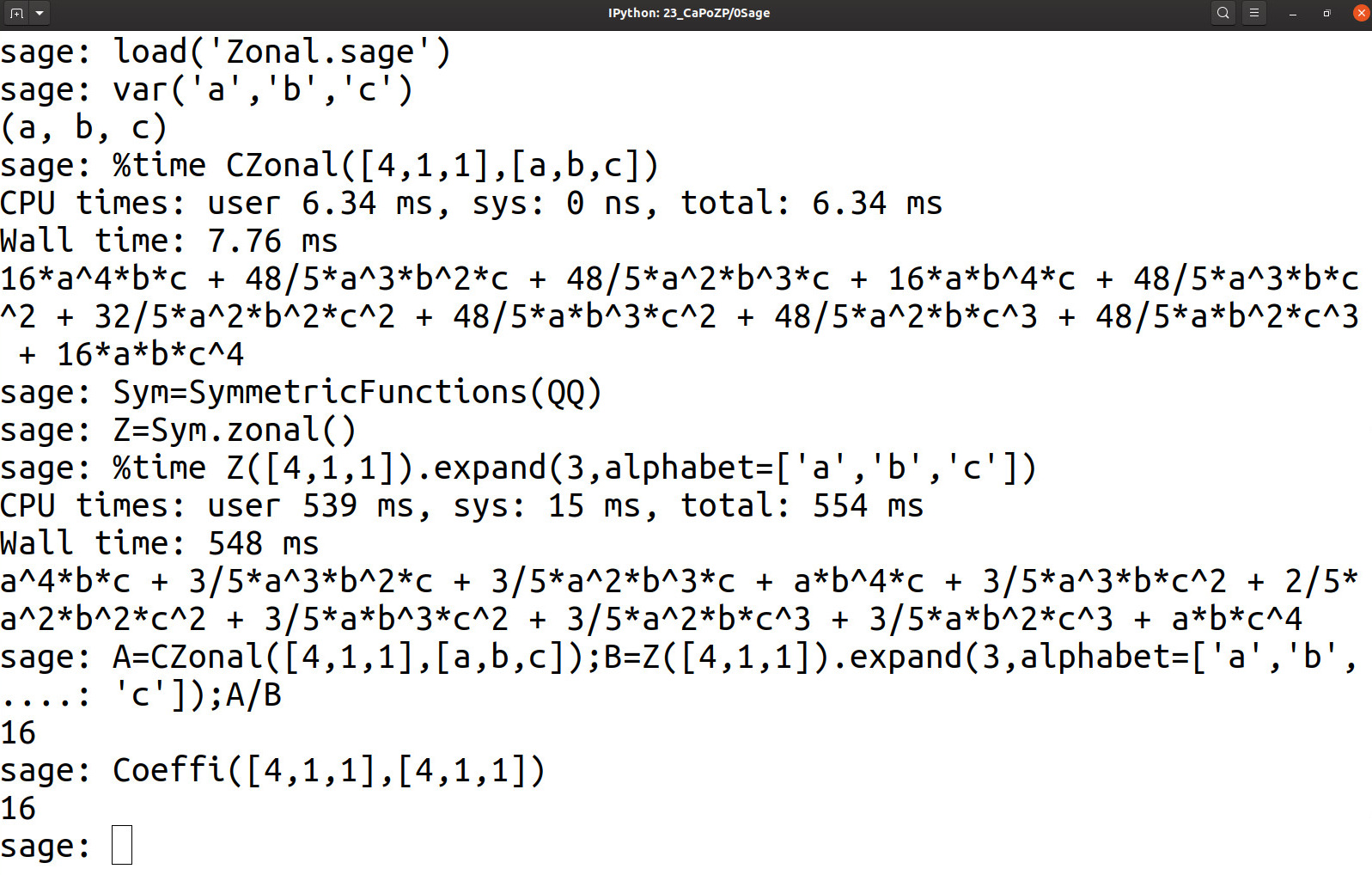}
\end{center}
\end{rem}

\section{Mathematica package for the calculation of zonal polynomials}\label{sec:mma}

We briefly describe our implementation of zonal polynomials in Mathematica.
Our software package \texttt{ZonalPolynomials.m} is freely available on
the website \url{www.koutschan.de/data/zonal/}. There we also provide a
demo notebook with further examples of its usage. The main functions of the
package are the following:
\medskip

\noindent\textbf{ZonalPolynomial}$[\lambda]$
gives the zonal polynomial indexed by the partition~$\lambda$,
in terms of the monomial symmetric functions~$M_\lambda$.
\begin{mma}
  \In |ZonalPolynomial|[\{3, 2\}] \\
  \Out \frac{48}{7}\,|M|[3, 2] + \frac{176}{21}\,|M|[2, 2, 1] +
    \frac{32}{7}\,|M|[3, 1, 1] + \frac{64}{7}\,|M|[2, 1, 1, 1] +
    \frac{80}{7}\,|M|[1, 1, 1, 1, 1] \\
\end{mma}

\noindent\textbf{ZonalPolynomial}$[\lambda, \{y_1,\dots,y_m\}]$
gives the zonal polynomial indexed by the partition~$\lambda$,
as a symmetric polynomial in the variables $y_1,\dots,y_m$.
\begin{mma}
  \In |ZonalPolynomial|[\{2,1\}, \{a,b,c\}] \\
  \Out \frac{12}{5}\,\bigl(a^2 b+a^2 c+a b^2+a c^2+b^2 c+b c^2\bigr)+\frac{18}{5} a b c \\
\end{mma}

\noindent\textbf{ZonalCoefficient}$[\kappa, \lambda]$
computes the zonal polynomial coefficient~$c_{\kappa,\lambda}$ recursively
by using \eqref{eq:CRec} and \eqref{eq:CInitial}.
\begin{mma}
  \In |ZonalCoefficient|[\{8, 6, 6, 3\}, \{7, 7, 5, 3, 1\}] // |Timing| \\
  \Out \left\{4.60311, \frac{33426505728}{5} \right\} \\
\end{mma}

\noindent\textbf{ZonalCoefficientTable}$[n]$
generates a table with all zonal polynomial coefficients~$c_{\kappa,\lambda}$,
where $\kappa$ and $\lambda$ are partitions of~$n$.
\begin{mma}
  \In |ZonalCoefficientTable|[4] \\
  \Out \biggl\{\Bigl\{1,\frac{4}{7},\frac{18}{35},\frac{12}{35},\frac{8}{35}\Bigr\},
    \Bigl\{0,\frac{24}{7},\frac{16}{7},\frac{88}{21},\frac{32}{7}\Bigr\},
    \Bigl\{0,0,\frac{16}{5},\frac{32}{15},\frac{16}{5}\Bigr\},
    \Bigl\{0,0,0,\frac{16}{3},\frac{64}{5}\Bigr\}, \linebreak
    \phantom{\biggl\{}\Bigl\{0,0,0,0,\frac{16}{5}\Bigr\}\biggr\} \\
\end{mma}

\noindent\textbf{ZonalCoefficientN}$[\kappa, \lambda]$
gives a symbolic expression (a rational function in~$n$) for the
zonal coefficient $c_{\kappa,\lambda}$ in the upper left corner.
The partitions $\kappa$ and $\lambda$ must be of the form $(n-i,\pi)$,
where $\pi$ is a partition of~$i$ and where $n$ is symbolic.
\begin{mma}
  \In |ZonalCoefficientN|[\{n-3,2,1\},\{n-4,2,2\}] \\
  \Out \frac{4 (n-3) (n-1) n (2 n^2-18 n+39)}{5 (2 n-11) (2 n-7)} \\
\end{mma}

\paragraph*{Acknowledgment}
We are grateful to Akimichi Takemura for helpful comments,  to Yi Zhang for
stimulating discussions, and to Tom Koornwinder for his comments on the relation between Macdonald polynomials and Jack polynomials. 
The first author also would like to thank Raymond Kan, who, after the
submission of the original draft, contributed to the SageMath program and now
becomes a co-contributor.
Both authors were supported by the Austrian Science Fund (FWF): P29467-N32, and the second author was also supported by F5011-N15.

\bibliographystyle{plain}
\bibliography{zonal}

\begin{thebibliography}{10}

\bibitem{DLMF}
Nist digital library of mathematical functions.
\newblock http://dlmf.nist.gov/, Release 1.0.20 of 2018-09-15.
\newblock F.~W.~J. Olver, A.~B. {Olde Daalhuis}, D.~W. Lozier, B.~I. Schneider,
  R.~F. Boisvert, C.~W. Clark, B.~R. Miller and B.~V. Saunders, eds.

\bibitem{ButlerPaige11}
Ronald Butler and Robert Paige.
\newblock Exact distributional computations for {R}oy's statistic and the
  largest eigenvalue of a {W}ishart distribution.
\newblock {\em Statistics and Computing}, 21:147--157, 2011.

\bibitem{GrossRichards87}
Kenneth Gross and Donald Richards.
\newblock Special functions of matrix argument. {I}. {A}lgebraic induction,
  zonal polynomials, and hypergeometric functions.
\newblock {\em Transactions of the AMS}, 301(2):781--811, 1987.

\bibitem{Hashiguchi}
Hiroki Hashiguchi, Yasuhide Numata, Nobuki Takayama, and Akimichi Takemura.
\newblock The holonomic gradient method for the distribution function of the
  largest root of a {W}ishart matrix.
\newblock {\em Journal of Multivariate Analysis}, 117:296--312, 2013.

\bibitem{Helgason}
Sigurdur Helgason.
\newblock {\em Differential Geometry and Symmetric Spaces}.
\newblock Academic Press, 1962.

\bibitem{Representation}
R.~Guti\'{e}rrez J\'{a}imez and J.~A.~Mermoso Guti\'{e}rrez.
\newblock An application of zonal polynomials to the generalization of
  probability distributions.
\newblock {\em Linear Algebra and its Applications}, 121:610--616, 1989.

\bibitem{James1}
A.~James.
\newblock Calculation of zonal polynomial coefficients by use of the
  {L}aplace-{B}eltrami operator.
\newblock {\em The Annals of Mathematical Statistics}, 39(5):1711--1718, 1968.

\bibitem{Johnstone01}
Iain Johnstone.
\newblock On the distribution of the largest eigenvalue in principal components
  analysis.
\newblock {\em The Annals of Statistics}, 29(2):295--327, 2001.

\bibitem{Guess}
Manuel Kauers.
\newblock Guessing handbook.
\newblock Technical Report 09-07, RISC Report Series, Johannes Kepler
  University, Linz, Austria, 2009.
\newblock
  http:/$\!$/www.risc.jku.at/\linebreak[0]research/\linebreak[0]combinat/\linebreak[0]software/\linebreak[0]Guess/.

\bibitem{KoevEdelman06}
Plamen Koev and Alan Edelman.
\newblock The efficient evaluation of the hypergeometric function of a matrix
  argument.
\newblock {\em Mathematics of Computation}, 75(254):833--846, 2006.

\bibitem{HolonomicFunctions}
Christoph Koutschan.
\newblock {HolonomicFunctions} (user's guide).
\newblock Technical Report 10-01, RISC Report Series, Johannes Kep\-ler
  University, Linz, Austria, 2010.
\newblock
  http:/$\!$/www.risc.jku.at/\linebreak[0]research/\linebreak[0]combinat/\linebreak[0]software/\linebreak[0]HolonomicFunctions/.

\bibitem{KushnerMeisner}
H.~B. Kushner and Morris Meisner.
\newblock Formulas for zonal polynomials.
\newblock {\em Journal of Multivariate Analysis}, 14:336--347, 1984.

\bibitem{Moakher}
Maher Moakher.
\newblock A differential geometric approach to the geometric mean of symmetric
  positive-definite matrices.
\newblock {\em SIAM Journal on Matrix Analysis and Applications},
  26(3):735--747, 2005.

\bibitem{Muirhead70}
Robb Muirhead.
\newblock Systems of partial differential equations for hypergeometric
  functions of matrix argument.
\newblock {\em The Annals of Mathematical Statistics}, 41(3):991--1001, 1970.

\bibitem{Muirhead}
Robb Muirhead.
\newblock {\em Aspects of multivariate statistical theory}.
\newblock Wiley series in probability and mathematical statistics.
  {P}robability and mathematical statistics. John Wiley \& Sons, New York,
  1982.

\bibitem{NakayamaEtAl11}
Hiromasa Nakayama, Kenta Nishiyama, Masayuki Noro, Katsuyoshi Ohara,
  Tomo{\-}nari Sei, Nobuki Takayama, and Akimichi Takemura.
\newblock Holonomic gradient descent and its application to the
  {F}isher-{B}ingham integral.
\newblock {\em Advances in Applied Mathematics}, 47(3):639--658, 2011.

\bibitem{Noro16}
Masayuki Noro.
\newblock System of partial differential equations for the hypergeometric
  function {1F1} of a matrix argument on diagonal regions.
\newblock In {\em Proceedings of the International Symposium on Symbolic and
  Algebraic Computation (ISSAC)}, ISSAC '16, pages 381--388, New York, NY, USA,
  2016. ACM.

\bibitem{SiriteanuEtAl15}
Constantin Siriteanu, Akimichi Takemura, Satoshi Kuriki, Hyundong Shin, and
  Christoph Koutschan.
\newblock {MIMO} zero-forcing performance evaluation using the holonomic
  gradient method.
\newblock {\em IEEE Transactions on Wireless Communications}, 14(4):2322--2335,
  2015.

\bibitem{Stembridge95}
John Stembridge.
\newblock A {M}aple package for symmetric functions.
\newblock {\em Journal of Symbolic Computation}, 20:755--768, 1995.

\bibitem{Takemura}
Akimichi Takemura.
\newblock {\em Zonal Polynomials}, volume~4 of {\em Institute of Mathematical
  Statistics Lecture Notes -- Monograph Series}.
\newblock Institute of Mathematical Statistics, Hayward, CA, 1984.

\bibitem{Wishart28}
John Wishart.
\newblock The generalised product moment distribution in samples from a normal
  multivariate population.
\newblock {\em Biometrika}, 20A(1-2):32--52, 1928.

\end{thebibliography}

\end{document}